\documentclass{amsart} 

\numberwithin{equation}{section}
\newtheorem{theorem}[equation]{Theorem}
\newtheorem{lemma}[equation]{Lemma}
\newtheorem{proposition}[equation]{Proposition}

\newtheorem{example}[equation]{Example}
\newcommand{\vect}[1]{\boldsymbol{#1}}

\newcommand{\n}{\mathbf{n}}

\newcommand{\rn}{\mathbf{R}^n}

\newcommand{\Pm}{\mathcal{P}_m}

\newcommand{\hm}{\mathcal{H}_m}

\newcommand{\be}{\partial E}

\renewcommand{\phi}{\varphi}

\newcommand{\Ref}[1]{(\ref{#1})}

\newcommand{\andd}{\quad \text{and} \quad}

\newcounter{subtheoremc}

\newenvironment{subtheorem}{\par\begin{list}{{\normalfont
(\alph{subtheoremc})}\hfill}{\usecounter{subtheoremc}
\setlength{\labelwidth}{.25in}
\setlength{\labelsep}{0pt}
\setlength{\leftmargin}{.25in}
}}{\end{list}}

\DeclareMathOperator{\ran}{range}
\DeclareMathOperator{\Null}{null}

\begin{document}

\title{The Neumann Problem on Ellipsoids}

\author{Sheldon Axler}
\address{Department of Mathematics,
San Francisco State University,
San Francisco, CA 94132 USA}
\email{axler@sfsu.edu}

\author{Peter J. Shin}
\address{Department of Radiology and Biomedical Imaging,
University of California, San Francisco,
San Francisco, CA 94158 USA}
\email{peter.shin@ucsf.edu}

\thanks{Peter J. Shin was supported by National Institutes of Health grant number P41-EB013598.}


\date{1 November 2016}

\maketitle

\begin{abstract}
The Neumann problem on an ellipsoid in $\rn$ asks for a function harmonic inside the ellipsoid whose normal derivative is some specified function on the ellipsoid. We solve this problem when the specified function on the ellipsoid is a normalized polynomial (a polynomial divided by the norm of the normal vector arising from the definition of the ellipsoid). Specifically, we give a necessary and sufficient condition for a solution to exist, and we show that if a solution exists then it is a polynomial whose degree is at most the degree of the polynomial giving the specified function. Furthermore, we give an algorithm for computing this solution. We also solve the corresponding generalized Neumann problem and give an algorithm for computing its solution.
\end{abstract}

\section{Introduction}

Fix a positive integer $n \ge 2$ and positive numbers $\beta_1, \dots, \beta_n$. Let $q$ be the function defined on $\rn$ by
\[
q(x) = q(x_1, \dots, x_n) = \beta_1 {x_1}^2 + \cdots + \beta_n {x_n}^2.
\]
Let $E$ be defined by
\[
E = \{x \in \rn : q(x) < 1\}.
\]
Thus the boundary of $E$, denoted $\partial E$, is the ellipsoid defined by
\[
\be = \{ x \in \rn: q(x) = 1\},
\]
and the closure of $E$, denoted $\bar{E}$, is defined by
$\bar{E} = \{x \in \rn : q(x) \le 1\}$.

Let $\n(x)$ be the outward-pointing unit normal on $\be$ at $x \in \be$. Thus
\[
\n(x) = \frac{\nabla q(x)}{\| \nabla q(x) \|},
\]
where
\[
(\nabla q)(x) = 2( \beta_1 x_1, \dots,  \beta_n x_n).
\]
The outward-pointing normal derivative at $x \in \be$ of a smooth function $h$ on $\bar{E}$, denoted $(D_\n h)(x)$, is given by the formula
\[
D_\n h = \nabla h \cdot \n = \nabla h \cdot \frac{ \nabla q }{\| \nabla q \|},
\]
where each of the functions above should be evaluated at $x \in \be$.

The Neumann problem on the ellipsoid $\be$ asks the following: given a function $u$ on $\be$, find a function $h$ harmonic on $\bar{E}$ such that $D_\n h = u$ on $\be$.

In this paper, we will solve the Neumann problem on the ellipsoid $\be$ when $u$ has the form $\dfrac{f}{\|\nabla q \|}$, where $f$ is a polynomial on $\rn$. Specifically, we will give a necessary and sufficient condition for this Neumann problem to have a solution, and when a solution exists we will show that it is a polynomial on $\rn$ with degree at most the degree of $f$ (Theorem \ref{exist}). Then we will give an algorithm for computing this solution (see Section \ref{algorithm}). This algorithm has been implemented in software, producing some beautiful examples (see Section \ref{examples}). We also solve the corresponding generalized Neumann problem, which instead of asking for $h$ to be harmonic asks for the Laplacian of $h$ to be some specified polynomial (still with $D_\n h = u$ on $\be$).

These results and an efficient algorithm for computing the solution to the Neumann problem with polynomial functions were known in the special case when the ellipsoid $\be$ is a sphere (see \cite{AR}). However, the results on the sphere do not translate to ellipsoids because the composition of a harmonic function with the natural linear map from $\rn$ to $\rn$ that takes a sphere to an ellipsoid is usually not a harmonic function.

The standard Dirichlet problem is closely related to the Neumann problem. On the sphere, the solution to the Dirichlet problem for polynomials leads to the solution to the Neumann problem for polynomials (see \cite{AR}). An algorithm for computing the solution to the Dirichlet problem for polynomials on ellipsoids was presented in \cite{AGV}. However, unlike the case of the sphere, the solution to the Dirichlet problem on ellipsoids does not seem to lead to a solution to the Neumann problem on ellipsoids.

One of the authors of this paper works in a Department of Radiology and Biomedical Imaging, which is an unusual affiliation for an author of a paper in a mathematics journal. The interesting mathematical questions answered in this paper arose from work in his lab with magnetic resonance imaging (MRI) scanners, a widely used tool in medical diagnostics. Section \ref{MRI} explains this connection between MRI and the Neumann problem on ellipsoids.

\section{Normal Derivatives of Polynomials on Ellipsoids}

The next proposition is well known, but we include it here for completeness. This proposition will show that the solution to our Neumann problem, if it exists, is unique except possibly for the addition of a constant function.

Let $dV$ denote the usual volume measure on $\rn$ and let $dA$ denote the usual surface area measure on $\be$.

\begin{proposition} \label{unique}
Suppose $h$ is harmonic on $\bar{E}$ and $\nabla h \cdot \nabla q  = 0$ on $\be$. Then $h$ is a constant function.
\end{proposition}

\begin{proof}
Green's First Identity states that
\[
\int_E (u \Delta h  + \nabla u \cdot \nabla h) \, dV = \int_{\be} u D_\n h \, dA,
\]
where $u$ and $h$ are smooth on $\bar{E}$, and $\Delta = \nabla^2$ is the Laplace operator. Take $u = h$ in the equation above. We have $\Delta h = 0$ (because $h$ is harmonic) and $D_\n h = 0$ (because $\nabla h \cdot \nabla q  = 0$). Thus the equation above becomes
\[
\int_E |\nabla h|^2 \, dV = 0.
\]
Hence $\nabla h = 0$ on $E$, which implies that $h$ is a constant function.
\end{proof}

The hypothesis in Proposition \ref{unique} that $h$ is harmonic cannot be relaxed to the hypothesis that $h$ is a polynomial. For example, take $n = 2$ and $q(x_1, x_2) = {x_1}^2 + {x_2}^2$. Let
\[
h(x) = {x_1}^4 + 2{x_1}^2 {x_2}^2 + {x_2}^4 - 2 {x_1}^2 - 2{x_2}^2.
\]
Then $(\nabla q)(x) = (2x_1, 2x_2)$ and
\[
(\nabla h)(x) = (4 {x_1}^3 + 4 x_1 {x_2}^2 - 4x_1, 4 {x_1}^2 x_2 + 4{x_2}^3 - 4 x_2).
\]
Thus
\begin{align*}
(\nabla h)(x) \cdot (\nabla q)(x) &= 8 {x_1}^4 + 8 {x_1}^2 {x_2}^2 - 8{x_1}^2 + 8 {x_1}^2 {x_2}^2 + 8{x_2}^4 - 4 {x_2}^2 \\
&= 8( {x_1}^2 + {x_2}^2 )^2 - 8( {x_1}^2 + {x_2}^2 ).
\end{align*}
On $\be$, both terms in parentheses above equal $1$. Thus $\nabla h \cdot \nabla q = 0$ on $\be$ even though $h$ is not a constant function.

For $m$ a nonnegative integer, let $\Pm$ denote the vector space of polynomials (with real coefficients) of degree at most $m$ on~$\mathbf{R}^n$. Let $\hm$ denote the subspace of $\Pm$ consisting of harmonic polynomials of degree at most $m$ on~$\mathbf{R}^n$. Let $\Pm|_{\be}$ denote the vector space of restrictions to $\be$ of functions in $\Pm$.

A multi-index $\alpha = (\alpha_1, \dots, \alpha_n)$ is an $n$-tuple of nonnegative integers. We define $|\alpha|$ by the equation
\[
|\alpha| = \alpha_1 + \cdots + \alpha_n.
\]
For $x = (x_1, \dots, x_n) \in \rn$, we let $x^\alpha$ denote the monomial ${x_1}^{\alpha_1} \cdots {x_n}^{\alpha_n}$, which has degree $|\alpha|$. For $m$ a nonnegative integer, $\Pm$ is obviously the span of $\{x^\alpha : |\alpha| \le m\}$.

The next theorem gives a necessary and sufficient condition for a solution to our Neumann problem on an ellipsoid to exist. The implications (b) $\Rightarrow$ (c) and (c) $\Rightarrow$ (a) in the theorem below are easy. The depth in this result is the implication (a) $\Rightarrow$ (b).

The proof given below that (a) implies (b) is an existence proof, provided by the magic of linear algebra. The proof provides no hint as to how to compute the harmonic polynomial $h$ satisfying (b) given a polynomial $f$ satisfying (a). In Section \ref{algorithm}, we will provide an algorithm for doing this computation.

\begin{theorem} \label{exist}
Suppose $f$ is a polynomial on $\rn$. Then the following are equivalent:
\begin{subtheorem}
\item
$\displaystyle
\int_{\be} \frac{f}{\| \nabla q \|} \, dA = 0$.
\medskip
\item
There exists a harmonic polynomial $h$ on $\rn$ with $\deg h \le \deg f$ such that
\[
\nabla h \cdot \nabla q = f \text{\ on\ } \be.
\]
\item
There exists a harmonic function $h$ on $\bar{E}$ such that
\[
D_\n h  = \frac{f}{\|\nabla q\|} \text{\ on\ } \be.
\]
\end{subtheorem}
\end{theorem}

\begin{proof}
First suppose that (b) holds. Because $\displaystyle D_\n h = \nabla h \cdot\frac{  \nabla q }{\| \nabla q \|}$, we see that (c) holds. Thus (b) implies (c).

Now suppose that (c) holds. Green's Second Identity states that if $g$ and $h$ are smooth functions on $\bar{E}$, then
\[
\int_E (g \Delta h - h \Delta g) \, dV = \int_{\be} (g D_\n h - h D_\n g) \, dA.
\]
In the equation above, take $g = 1$; thus $\Delta g = 0$ and $D_\n g = 0$. Our function $h$ provided by (c) is harmonic, and thus $\Delta h = 0$. Hence the equation above becomes
\begin{align*}
0 &= \int_{\be} D_\n h \, dA \\[4bp]
&= \int_{\be} \frac{f}{\| \nabla q \|} \, dA
\end{align*}
which completes the proof that (c) implies (a).

To prove that (a) implies (b), now  suppose that (a) holds. Let $m = \deg f$. Define linear maps $T \colon \hm \to \Pm|_{\be}$ and $U \colon \hm \to \Pm|_{\be}$ by
\[
T(h) = (\nabla h \cdot \nabla q)|_{\be} \andd U(h) = h|_{\be}.
\]
Taking a partial derivative reduces the degree of a polynomial by $1$, and then taking the dot product with $\nabla q = (2 \beta_1 x_1, \dots,  2\beta_n x_n)$ increases the degree back by $1$ (unless $h$ is a constant function). Thus $\deg (\nabla h \cdot \nabla q) = \deg h$ for all nonconstant functions $h \in \hm$. In other words, $T$ really does map $\hm$ into $\Pm|_{\be}$.

Proposition \ref{unique} tells us that $\Null T$, the null space of $T$, is the set of constant functions. Thus $\dim \Null T = 1$. A wonderful theorem that appears in every linear algebra book states that for a linear map, the dimension of the domain equals the dimension of the range plus the dimension of the null space. Thus
\begin{equation} \label{dimranT}
\dim \ran T = (\dim \hm) - 1.
\end{equation}

If $h \in \hm$ and $h|_{\be} = 0$ then the maximum principle for harmonic functions implies that $h = 0$ (for example, see 1.9 in \cite{HFT}). Thus $U$ is injective.

The range of $U$ is all of $\Pm|_{\be}$ because the Dirichlet problem with polynomial data for ellipsoids has polynomial solutions without increasing the degree; see, for example, Fishers's Decomposition Theorem (2.2 in \cite{AGV}) or Theorem 1 in \cite{Baker}.

Because $U$ is both injective and surjective, we can conclude that
\begin{equation} \label{dimhm}
\dim \hm = \dim \Pm|_{\be}.
\end{equation}

Define a linear functional $\phi \colon  \Pm|_{\be} \to \mathbf{R}$ by
\[
\phi(g) = \int_{\be} \frac{g}{\|\nabla q\|} \, dA.
\]
Because $\phi$ is a nonzero linear functional, we have
\begin{equation} \label{dimnullphi}
\dim \Null \phi = (\dim \Pm|_{\be}) - 1
\end{equation}

We have already proved that (b) $\Rightarrow$ (c) $\Rightarrow$ (a). In particular, (b) $\Rightarrow$ (a), which implies that
\[
\ran T \subset \Null \phi.
\]
Now \Ref{dimranT}, \Ref{dimhm}, and \Ref{dimnullphi} imply that the two subspaces above have the same dimension. Thus we have
\begin{equation} \label{rangenull}
\ran T = \Null \phi.
\end{equation}

Our hypothesis (a) implies that $f|_{\be} \in \Null \phi$. Thus \Ref{rangenull} implies that $f \in \ran T$. Hence there exists $h \in \hm$ such that $\nabla h \cdot \nabla q = f$ on $\be$. In other words, (b) holds, completing the proof that (a) implies (b).
\end{proof}

The generalized Neumann problem for an ellipsoid asks the following: Given polynomials $f$ and $g$ on $\rn$, find a polynomial $h$ on $\rn$ such that $D_{\n} h = \dfrac{f}{\|\nabla q\|}$ on $\be$ and $\Delta h = g$. If $g = 0$, then this generalized Neumann problem asks for $h$ to be harmonic, and thus it is then the Neumann problem we have already discussed.

Our solution to the generalized Neumann problem will require the following lemma, which is well known. A proof of the lemma below can be obtained by considering the linear map $u \mapsto \Delta u$ from $\mathcal{P}_{m+2}$ to $\mathcal{P}_m$; the null space of this map is $\mathcal{H}_{m+2}$; counting dimensions of the various spaces (use Proposition 5.8 in [ABR]) shows that this map is onto $\mathcal{P}_m$. However, the proof just outlined gives no hint as to how to calculate $u$ (which is not unique) given $g$ in the lemma below. Thus we present a constructive proof because our algorithm for solving the generalized Neumann problem will require a way to compute an antiLaplacian of a polynomial.

\begin{lemma} \label{Anti}
Suppose $g$ is a polynomial on $\rn$. Then there exists a polynomial $u$ on $\rn$ such that $\deg u = 2 + \deg g$ and $\Delta u = g$.
\end{lemma}

\begin{proof}
It suffices to consider the case where $g$ is a monomial. Thus suppose that $g(x) = x^\alpha$ for some multi-index $\alpha$.

It is easy to see that
\begin{equation} \label{anti}
\Delta\Bigl( \frac{ {x_1}^2 x^\alpha }{ (\alpha_1 + 1)(\alpha_1 + 2) } \Bigr) = x^\alpha + \sum_{k=2}^n \frac{ (\alpha_k - 1) \alpha_k }{  (\alpha_1 + 1)(\alpha_1 + 2) }  \frac{ {x_1}^2 x^\alpha }{ {x_k}^2 }.
\end{equation}
The coefficient $\dfrac{ (\alpha_k - 1) \alpha_k }{  (\alpha_1 + 1)(\alpha_1 + 2) }$ equals $0$ if $\alpha_k = 1$ or $\alpha_k = 0$; thus the expression on the right is a polynomial even though it looks more like a rational function.

The equation above reduces the problem of finding an antiLaplacian of $x^\alpha$ to the problem of finding an antiLaplacian of each term in the summation on the right side of \Ref{anti}. In other words, we have a new set of antiLaplacian problems, where the original $\alpha_1$ has been replaced by $\alpha_1 + 2$ and an $\alpha_k \ge 2$ has been replaced by $\alpha_k - 2$. Iterating this process, we eventually reduce the problem to computing an antiLaplacian of $x^\alpha$ to the special case where $\alpha_k \in \{0, 1\}$ for each $k \ge 2$. In that case, \Ref{anti} shows that $\dfrac{ {x_1}^2 x^\alpha }{ (\alpha_1 + 1)(\alpha_1 + 2) }$ is an antiLaplacian of $x^\alpha$, completing the proof.
\end{proof}

For example, the algorithm provided by the proof above quickly finds that an antiLaplacian of the degree $14$ monomial ${x_1}^9 {x_2}^3 {x_3}^2$ is the degree $16$ polynomial
\[
\frac{ 2 {x_1}^{15} x_2 - 35 {x_1}^{13} {x_2}^3 - 105 {x_1}^{13} x_2 {x_3}^2 + 2730 {x_1}^{11} {x_2}^3 {x_3}^2 }{ 300300}.
\]

The special case of the next result when $g = 0$ is just Theorem \ref{exist}. We cannot eliminate Theorem \ref{exist} and just prove the theorem below because the proof of the theorem below requires Theorem \ref{exist}.

Unlike the proof of Theorem \ref{exist}, the proof below provides an algorithm for computing  the solution to generalized Neumann problems, provided that we can compute the solution to the regular Neumann problem (which we will show how to do in Section \ref{algorithm}).

\begin{theorem} \label{generalized}
Suppose $f$ and $g$ are polynomials on $\rn$. Then the following are equivalent:
\begin{subtheorem}
\item
$\displaystyle
\int_{\be} \frac{f }{\| \nabla q \|} \, dA = \int_E g \, dV$.
\medskip
\item
There exists a polynomial $h$ on $\rn$ with $\deg h \le \max\{\deg f, 2 + \deg g\}$ such that
\[
\Delta h = g \andd \nabla h \cdot \nabla q = f \text{\ on\ } \be.
\]
\item
There exists a smooth function $h$ on $\bar{E}$ such that
\[
\Delta h = g \andd D_\n h  = \frac{f}{\|\nabla q\|} \text{\ on\ } \be.
\]
\end{subtheorem}
\end{theorem}

\begin{proof}
First suppose that (b) holds. Because $\displaystyle D_\n h = \nabla h \cdot\frac{  \nabla q }{\| \nabla q \|}$, we see that (c) holds. Thus (b) implies (c).

Now suppose that (c) holds. Green's Second Identity states that if $u$ and $h$ are smooth functions on $\bar{E}$, then
\[
\int_E (u \Delta h - h \Delta u) \, dV = \int_{\be} (u D_\n h - h D_\n u) \, dA.
\]
In the equation above, take $u = 1$; thus $\Delta u = 0$ and $D_\n u = 0$. Our function $h$ provided by (c) satisfies the equation $\Delta h = g$. Hence the equation above becomes
\begin{align*}
\int_E g \, dV &= \int_{\be} D_\n h \, dA \\[4bp]
&= \int_{\be} \frac{f}{\| \nabla q \|} \, dA,
\end{align*}
which completes the proof that (c) implies (a).

To prove that (a) implies (b), now  suppose that (a) holds. Let $u$ be a polynomial on $\rn$ with $\deg u = 2 + \deg g$ and $\Delta u = g$; the existence of this antiLaplacian $u$ is guaranteed by Lemma \ref{Anti}.

Now
\begin{align*}
\int_{\be} \frac{  \nabla u \cdot \nabla q}{\| \nabla q \|} \, dA &= \int_{\be} D_{\n}u \, dA \\[3bp]
&= \int_E \Delta u \, dV \\[3bp]
&= \int_E g \, dV\\[3bp]
&= \int_{\be} \frac{ f }{\| \nabla q \|} \, dA,
\end{align*}
where the second equality comes from Green's Second Identity (take one of the functions to equal $1$) and the last equality comes from the assumption in (a).

We now use the implication (a) $\Rightarrow$ (b) in Theorem \ref{exist} with $f$ in Theorem \ref{exist} replaced by $f - \nabla u \cdot \nabla q$, which is valid because the equation above tells us that $\displaystyle\int_{\be} \frac{ f - \nabla u \cdot \nabla q}{\| \nabla q \|} \, dA = 0$. The degree of $f - \nabla u \cdot \nabla q$ is at most $\max\{\deg f, 2 + \deg g\}$. Thus Theorem \ref{exist} implies that there exists a harmonic polynomial $p$ with degree at most  $\max\{\deg f, 2 + \deg g\}$ such that
\[
\nabla p \cdot \nabla q = f - \nabla u \cdot \nabla q \text{\ on\ } \be.
\]
Let $h = u + p$. Then $\deg h \le  \max\{\deg f, 2 + \deg g\}$ and
\[
\Delta h = \Delta u + \Delta p = g + 0 = g.
\]
Furthermore,
\[
\nabla h \cdot \nabla q = \nabla u \cdot \nabla q + \nabla p \cdot \nabla q = \nabla u \cdot \nabla q + (f - \nabla u \cdot \nabla q) = f \text{\ on\ } \be,
\]
completing the proof that (a) implies (b).
\end{proof}

\section{Computing Surface Area Integrals on an Ellipsoid}

We now turn to the question of computing $\displaystyle\int_{\be} \frac{f}{\| \nabla q \|} \, dA(x)$ for a polynomial $f$ on $\rn$. This question is of interest because Theorem \ref{exist} tells us that $\dfrac{f}{\|\nabla q\|}$ is the normal derivative on $\be$ of some harmonic polynomial on $\rn$ if and only if this integral equals $0$. Also, Proposition \ref{integratearea} below is used by the software described in Section \ref{examples}.

Each polynomial $f$ on $\rn$ can be written in the form $f = \sum_{\alpha} c_\alpha x^\alpha$ for some choice of constants $\{c_\alpha\}$. Hence we concentrate on computing $\displaystyle\int_{\be} \frac{x^\alpha}{\| \nabla q(x) \|} \, dA(x)$.

The double factorial will appear in our next result. If $m$ is an odd positive integer, then the double factorial of $m$, denoted $m!!$, is the product of the positive odd integers less than or equal to $m$. In other words,
\[
m!! = 1 \cdot 3 \cdot 5 \cdot \ \cdots \ \cdot m.
\]
For convenience, we define $(-1)!! = 1$.

Let $B$ denote the open unit ball in $\rn$; thus
\[
B = \{x \in \rn: \|x\| < 1 \}.
\]
The volume of $B$ is denoted by $\textup{vol}(B)$. Thus $\textup{vol}(B) = \frac{4}{3}\pi$ if $n = 3$; the formula for other values of $n$ is derived, for example, in Appendix A of \cite{HFT}.

If at least one of the nonnegative integers $\alpha_1, \dots, \alpha_n$ is odd, then it is easy to see that $\displaystyle\int_{\be} \frac{ x^\alpha }{\|\nabla q(x)\|} \, dA(x) = 0$. Thus the next result only considers the case where each $\alpha_j$ is even.

Let $\beta = (\beta_1, \dots, \beta_n)$. Because $\nabla q = 2(\beta_1 x_1, \dots, \beta_n x_n)$, we have the equation
\begin{equation} \label{monomap}
\nabla x^\alpha \cdot \nabla q = 2(\alpha \cdot \beta) x^\alpha,
\end{equation}
which will be used in the proof below and in the next section.

\begin{proposition} \label{integratearea}
Suppose $\alpha = (\alpha_1, \dots, \alpha_n)$ is an $n$-tuple of nonnegative even integers. Then
\[
\int_{\be} \frac{x^\alpha}{\| \nabla q(x) \|} \, dA(x) = \frac{n\,\textup{vol}(B)}{ 2 \sqrt{\prod_{j=1}^n {\beta_j}^{\alpha_j + 1}}}  \cdot \frac{ (\alpha_1 - 1)!! \cdots (\alpha_n - 1)!! }{n(n+2)\cdots (n + |\alpha| - 2) }.
\]
\end{proposition}

\begin{proof}
First we consider the case where $|\alpha| > 0$. We have
\begin{align*}
\int_{\be} \frac{x^\alpha}{\| \nabla q(x) \|} \, dA(x) &= \int_{\be} \frac{\nabla x^\alpha \cdot \nabla q}{2 (\alpha \cdot \beta) \|\nabla q(x)\|} \, dA(x)  \\[6bp]
&= \frac{1}{2 (\alpha \cdot \beta)} \int_{\be} D_\n x^\alpha \, dA(x) \\[6bp]
&= \frac{1}{2 (\alpha \cdot \beta)} \int_E \Delta(x^\alpha) \, dV(x),
\end{align*}
where the last equality follows from Green's Second Identity. Evaluating the Laplacian $\Delta(x^\alpha)$ we thus have
\begin{align*}
\int_{\be} \frac{x^\alpha}{\| \nabla q(x) \|} \, dA(x) &=  \frac{1}{2 (\alpha \cdot \beta)} \int_E \sum_{j=1}^n \alpha_j(\alpha_j -1 ) \frac{x^\alpha}{{x_j}^2} \, dV(x) \\[6bp]
&= \frac{1}
{2 (\alpha \cdot \beta) \sqrt{\prod_{j=1}^n {\beta_j}^{\alpha_j + 1}}}
\sum_{j = 1}^n \alpha_j (\alpha_j - 1) \beta_j  \int_B \frac{x^\alpha}{{x_j}^2} \, dV(x),
\end{align*}
where the last equation comes from a standard change of variables to change the integral from $E$ to the ball $B$.

Because $\displaystyle\int_B f(x) \, dV(x) = \int_0^1 r^{n-1} \int_{\partial B} f(rx) \, dA(x) \, dr$ for every continuous function $f$ on $B$ (see, for example, Exercise 6 in Chapter 8 of \cite{Rudin}), the equation above can be rewritten as
\begin{align*}
\int_{\be} \frac{x^\alpha}{\| \nabla q(x) \|}& \, dA(x) \\
&\hspace*{-.1in}= \frac{1}{2 (\alpha \cdot \beta) \sqrt{\prod_{j=1}^n {\beta_j}^{\alpha_j + 1}}(n + |\alpha| - 2)} \sum_{j = 1}^n \alpha_j (\alpha_j - 1) \beta_j  \int_{\partial B} \frac{x^\alpha}{{x_j}^2} \, dA(x).
\end{align*}
Using the formula for integrating a monomial over the unit sphere $\partial B$ (see Section 3 of Hermann Weyl's paper \cite{Weyl}), this becomes
\begin{align*}
\int_{\be} &\frac{x^\alpha}{\| \nabla q(x) \|} \, dA(x)\\
&= \frac{n \,\text{vol}(B)}{2 (\alpha \cdot \beta) \sqrt{\prod_{j=1}^n {\beta_j}^{\alpha_j + 1}}(n + |\alpha| - 2)} \sum_{j = 1}^n \alpha_j  \beta_j  \frac{ (\alpha_1 - 1)!! \cdots (\alpha_n - 1)!!}{ n(n+2) \cdots (n + |\alpha| - 4) } \\[6bp]
&= \frac{n \, \text{vol}(B)}{ 2 \sqrt{\prod_{j=1}^n {\beta_j}^{\alpha_j + 1}}}  \cdot \frac{ (\alpha_1 - 1)!! \cdots (\alpha_n - 1)!! }{n(n+2)\cdots (n + |\alpha| - 2) },
\end{align*}
completing the proof in the case where $|\alpha| > 0$.

Now suppose $|\alpha| = 0$. In other words, we want to compute $\displaystyle\int_{\be} \frac{1}{\|\nabla q(x)\|} \, dA(x)$. The constant function $1$ cannot be written in the form $\nabla f \cdot \nabla q$ for any polynomial $f$, and hence the technique used above when $|\alpha| > 0$ will not work. However, $1 = \beta_1 {x_1}^2 + \cdots + \beta_n {x_n}^2$ on $\be$. Thus
\begin{align*}
\int_{\be} \frac{1}{\|\nabla q(x)\|} \, dA(x) &= \int_{\be} \frac{\beta_1 {x_1}^2 + \cdots + \beta_n {x_n}^2}{\|\nabla q(x)\|} \, dA(x) \\[6bp]
&= \sum_{k=1}^n \beta_k \int_{\be} \frac{ {x_k}^2 }{\|\nabla q(x)\|} \, dA(x) \\[6bp]
&= \sum_{k=1} \beta_k \frac{n \, \text{vol}(B) } { 2 \beta_k \sqrt{\prod_{j=1}^n \beta_j}} \cdot \frac{1}{n}\\[6bp]
&= \frac{n \, \text{vol}(B) } { 2 \sqrt{\prod_{j=1}^n \beta_j}},
\end{align*}
where the third equality above comes from the formula we have already proved in the case when $|\alpha| > 0$. The formula above is the desired result when $|\alpha| = 0$ because in this case we interpret the empty product $n(n+2)\cdots (n + |\alpha| - 2)$ to be $1$.
\end{proof}

\section{An Algorithm for Solving the Neumann Problem} \label{algorithm}

In this section, we present an algorithm that solves our Neumann problem. More specifically, for a given polynomial $f$ satisfying condition (a) of Theorem \ref{exist}, our algorithm finds the unique harmonic polynomial $h$ such that $\nabla{h} \cdot \nabla{q} = f$ on $\partial{E}$ and $h(0) = 0$.

Let $\mathcal{P}^0$ denote the vector space of polynomials $g$ on $\rn$ such that $g(0) = 0$. To develop the algorithm, we define the linear map $S \colon \mathcal{P}^0 \to \mathcal{P}^0$ by
\[
	S(g) =\nabla{g} \cdot \nabla{q}.
\]
Each polynomial $g \in \mathcal{P}^{0}$ can be written as a finite linear combination of monomials
\[
	g=\sum_{0 < |\alpha|} c_{\alpha}x^{\alpha}
\]
for some constants $\{c_\alpha\}$ (where all but finitely many of the $c_\alpha$ equal $0$). From \Ref{monomap} we have
\[
S\Bigl(\sum_{0 < |\alpha|} c_{\alpha}x^{\alpha}\Bigr) = \sum_{0 <|\alpha|} 2(\alpha\cdot\beta) c_{\alpha}x^{\alpha}.
\]
The equation above shows that $S$ is a one-to-one mapping of $\mathcal{P}^0$ onto itself and has an inverse
\begin{equation} \label{Sinverse}
	S^{-1}\Bigl(\sum_{0 < |\alpha|} c_{\alpha}x^{\alpha}\Bigr)=\sum_{0 < |\alpha|} \frac{c_{\alpha}}{2(\alpha \cdot\beta)} x^{\alpha}.
\end{equation}
Furthermore, $\deg g = \deg S(g) = \deg S^{-1}(g)$ for each $g \in \mathcal{P}^0$.

Suppose $g = \sum_{|\alpha| \le m} c_\alpha x^\alpha \in \Pm$ . Then $qg \in \mathcal{P}^0$. Let $\delta_j$ denote the multi-index whose $j^{\text{th}}$-coordinate equals $1$ and all other coordinates equal $0$. We will need the following formula for our algorithm:
\begin{align}
	S^{-1}(qg) &= S^{-1}\Bigl(\sum_{j=1}^{n}\beta_j {x_j}^2\sum\limits_{|\alpha|\leq m}c_{\alpha}x^{\alpha}\Bigr) \notag\\
	&= S^{-1}(\sum\limits_{|\alpha|\leq m}\sum_{j=1}^{n}\beta_j c_{\alpha} x^{\alpha+2\delta_j})   \notag\\
	&= \sum\limits_{|\alpha|\leq m}\sum_{j=1}^{n} \frac{\beta_j c_{\alpha}}{2\bigl((\alpha+2\delta_j)\cdot\beta\bigr)} x^{\alpha+2\delta_j}   \notag\\
	&= \sum\limits_{|\alpha|\leq m}\sum_{j=1}^{n} \frac{\beta_j c_{\alpha}}{2(\alpha \cdot\beta)+4\beta_j} x^{\alpha+2\delta_j}. \label{sinvqg}
\end{align}

Now we are ready to develop the algorithm that solves our Neumann problem on ellipsoids. Suppose $f \in \mathcal{P}_{m+2}$ and $\displaystyle \int_{\be} \frac{f}{\|\nabla q\|} \, dA = 0$. Theorem \ref{exist} implies that there exists a harmonic polynomial $h \in \mathcal{H}_{m+2}$ with $h(0) = 0$ such that $\nabla h \cdot \nabla q = f$ on $\be$. Thus $\nabla h \cdot \nabla q - f$ is a polynomial that equals $0$ on  $\be$, which implies (see Lemma 2.9 of \cite{AGV}) that there exists a polynomial $g \in \Pm$ such that
\begin{equation} \label{maineq}
	f	= \nabla{h} \cdot \nabla{q} + (q-1)g.
\end{equation}
We need an algorithm to calculate $h$ given $f$ and $q$. Our plan of attack is to first calculate the polynomial $g$ in the equation above. We will do this by transforming the equation above to take the harmonic polynomial $h$ temporarily out of the calculation. After finding $g$, we will go back to the equation above to calculate $h$.

Let $f_k$ and $g_k$ be polynomials that are homogeneous of degree $k$ such that $f = \sum_{k=0}^{m+2} f_k$ and $g = \sum_{k=0}^m g_k$. Writing
\[
f	= \nabla{h} \cdot \nabla{q} + qg-g,
\]
we immediately see from \Ref{maineq} that $g_0=-f_0$.
If we subtract the constant terms $f_0$ and $-g_0$ from each side of the equation above, then we are left with polynomials in $\mathcal{P}^0$. We can then apply $S^{-1}$ to get
\begin{align}
S^{-1}(f-f_0)   &= S^{-1}(\nabla{h} \cdot \nabla{q}) + S^{-1}(qg) - S^{-1}(g-g_0) \notag \\
				&= h + S^{-1}(qg) - S^{-1}(g-g_0).	\label{Sinveq}			
\end{align}
The equation above, along with the explicit formulas \Ref{Sinverse} and \Ref{sinvqg} for $S^{-1}$, show that we can calculate $h$ once we know $g$.

To make the problem more tractable, we break each polynomial in \Ref{Sinveq} into its homogeneous components to obtain
\[
	S^{-1}\Bigl(\sum_{k=1}^{m+2}f_k\Bigr) = \sum_{k=1}^{m+2}h_k + S^{-1}\Bigl(q\sum_{k=0}^{m}g_k\Bigr) - S^{-1}\Bigl(\sum_{k=1}^{m}g_k\Bigr),
\]
where each $h_k$ is harmonic (see page 75 of \cite{HFT}). Since $S^{-1}$ is linear and preserves the degree of polynomials, we can further break the equation above into homogeneous equations by degree and obtain the following system of equations:
\begin{alignat}{3}
m+2&: \qquad \qquad \qquad S^{-1}(f_{m+2}) &&= h_{m+2} + S^{-1}(qg_m) \notag \\
m+1&: \qquad \qquad \qquad S^{-1}(f_{m+1}) &&= h_{m+1} + S^{-1}(qg_{m-1}) \notag \\
m&: \qquad \qquad \qquad S^{-1}(f_{m}) + S^{-1}(g_m) &&= h_{m} + S^{-1}(qg_{m-2}) \notag \\
m-1&: \qquad \qquad \qquad S^{-1}(f_{m-1}) + S^{-1}(g_{m-1}) &&= h_{m-1} + S^{-1}(qg_{m-3}) \notag \\
&&\vdots \label{E:syseqn}\\
3&: \qquad \qquad \qquad S^{-1}(f_{3}) + S^{-1}(g_3)&&= h_{3} + S^{-1}(qg_1) \notag \\
2&: \qquad \qquad \qquad S^{-1}(f_{2}) + S^{-1}(g_2)&&= h_{2} + S^{-1}(qg_0) \notag \\
1&: \qquad \qquad \qquad S^{-1}(f_{1}) + S^{-1}(g_1) &&= h_{1} \notag\\
0&: \qquad \qquad \qquad f_0 + g_0 &&= 0 \notag
\end{alignat}
Here, each $f_k$ and $q$ are known and we need to compute $h_{m+2}, \ldots, h_1$ and $g_m, \ldots, g_0$.

The first thing to note is that even and odd degree equations are decoupled from one another (a typical equation above involves $f_k$, $g_k$, and $g_{k-2}$). Hence, the problem at hand can be decomposed into two smaller sub-problems each exclusively involving either even or odd degree polynomials.

Additionally, if for convenience we set
\begin{equation} \label{Sm}
S^{-1}(g_{m+2})=S^{-1}(g_{m+1})=S^{-1}(qg_{-1})=0,
\end{equation}
then all the equations above have the same form
\begin{equation} \label{ktheq}
k+2: \qquad \qquad \qquad	S^{-1}(f_{k+2})+S^{-1}(g_{k+2})=h_{k+2}+S^{-1}(qg_{k}).
\end{equation}
Our strategy is to start from the highest degree equation and solve the system sequentially down to lower degree equations. More specifically, we will solve the first equation to find $g_m$, which we pass down to the left-hand-side of the $m^{th}$ stage equation to find $g_{m-2}$ in the right-hand-side and so on. We apply the same process to find $g_{m-1}$ and its successors. Hence, if we know how to solve \Ref{ktheq}, we can repetitively apply it to solve the system in \Ref{E:syseqn} and calculate all $g_k$ for $k=0,\ldots,m$. Once we have $g$, we plug it back into \Ref{Sinveq} to find $h$.

For convenience, define
\begin{equation} \label{rk}
	r_k = S^{-1}(f_k) + S^{-1}(g_k).
\end{equation}
Then, \Ref{ktheq} can be written as
\[
r_{k+2} = h_{k+2} + S^{-1}(q g_k).
\]
Write $g_k = \sum_{|\alpha|=k} c_{\alpha} x^{\alpha}$as a sum of monomials of degree $k$ with unknown coefficients \{$c_{\alpha}$\}.
Now apply the Laplacian operator to both sides of the equation above to take the harmonic polynomial $h_{k+2}$ out of the calculation:\pagebreak[0]
\begin{align} \label{syseq}
\Delta (r_{k+2})&= \Delta \bigl(S^{-1}(q g_k)\bigr) \notag \\
&= \Delta \Bigl({\sum\limits_{|\alpha| = k}\sum_{j=1}^{n} \frac{\beta_j c_{\alpha}}{2(\alpha \cdot \beta)+4\beta_j} x^{\alpha+2\delta_j}}\Bigr) \notag \\
&= {\sum\limits_{|\alpha| = k}\sum_{j=1}^{n} \frac{\beta_j c_{\alpha}}{2(\alpha \cdot \beta)+4\beta_j} \Delta (x^{\alpha+2\delta_j}}) \notag \\
&= \sum\limits_{|\alpha| = k}\sum_{j=1}^{n} \frac{\beta_j c_{\alpha}}{2(\alpha \cdot \beta)+4\beta_j}\Bigl(\sum\limits_{l=1}^{n}\alpha_{l}(\alpha_{l}-1) x^{\alpha+2\delta_j-2\delta_l} + (4\alpha_j+2)x^\alpha\Bigr),
\end{align}
where the second equality comes from \Ref{sinvqg}. Note that the constants {$c_\alpha$} are the only unknowns on the right side of the above equation. Once we find all these coefficients, we will know $g_k$.

Begin by considering the cases where $k = m$ or $k = m-1$. In those cases, \Ref{rk} and \Ref{Sm} imply that $r_{k+2} = S^{-1}(f_{k+2})$ and thus the left side of \Ref{syseq} is known. Hence we can solve for the $\{c_\alpha\}$ corresponding to $g_m$ and $g_{m-1}$ (as discussed in the next paragraph). Now that $g_m$ and $g_{m-1}$ are known, we can consider the cases where $k = m-2$ or $k = m-3$. In those cases, \Ref{rk} shows that $r_m$ and $r_{m-1}$ are known, and hence the left side of \Ref{syseq} is again known; thus we can solve for the $\{c_\alpha\}$ corresponding to $g_{m-2}$ and $g_{m-3}$  (as discussed in the next paragraph). This process can be continued, solving for $g_m, g_{m-1}, g_{m-2}, \dots, g_1, g_0$ and thus solving for $g$.

All that remains is to discuss how to solve \Ref{syseq} for the $\{c_\alpha\}$ in the case where we know the left side of \Ref{syseq}. Both sides of \Ref{syseq} are homogeneous polynomials of degree $k$. Thus by comparing the coefficients of monomials $x^\alpha$ ($|\alpha| = k$) on each side of the equation, we get a system of linear equations, which can be solved by Gaussian elimination for the $\{c_\alpha\}$ (a solution is guaranteed to exist by Theorem \ref{exist}).

The aforementioned system of linear equations with unknowns \{$c_\alpha$\} can be broken down into several smaller systems of linear equations by partitioning the multi-indices \{$\alpha : |\alpha|=k$\} into groups that have the same parities---two multi-indices $\alpha$ and $\gamma$ are grouped together if $\alpha_i = \gamma_i$ (mod 2) for each $i=1,\ldots,n$. Solving these several smaller systems of equations instead of the one large system results in significant computational savings, as discussed in \cite{AGV} in connection with solving the Dirichlet problem on ellipsoids.

The algorithm discussed in this section has been implemented in software, with results that can be verified to be correct. The next section provides examples that were calculated using the algorithm discussed here.

\section{Examples} \label{examples}

The results and algorithms in this paper have been incorporated into a new version of the \textit{HFT Mathematica} package for symbolic manipulation of harmonic functions. This software is available without charge from the websites listed at \cite{A}. The \texttt{neumann} section of the \textit{Computing with Harmonic Functions} documentation available at \cite{A} is particularly relevant to this paper.

As an example of the Neumann problem on an ellipsoid in $\mathbf{R}^3$, we start with the function ${x_1}^4 {x_2}^2$ on the ellipsoid $\{x \in \mathbf{R}^3 : 3{x_1}^2 + {x_2}^2 + 2{x_3}^2 = 1\}$. This function does not satisfy the necessary condition (a) of Theorem \ref{exist}, but we can easily adjust it by adding an appropriate constant. Specifically, Proposition \ref{integratearea} can be used to show that for the ellipsoid under consideration, ${x_1}^4 {x_2}^2 - \frac{1}{315}$ satisfies condition (a) of Theorem \ref{exist} (the \textit{HFT Mathematica} package can perform this calculation using its \texttt{integrateEllipsoidArea} function). Then the \texttt{neumann} function in the \textit{HFT Mathematica} package, which uses the algorithm described in Section \ref{algorithm} of this paper, produces the following result.

\begin{example} \label{example}
Suppose $f(x_1, x_2, x_3) = {x_1}^4 {x_2}^2 - \frac{1}{315}$ and
\[
q(x) = 3{x_1}^2 + {x_2}^2 + 2{x_3}^2.
\]
Then the degree $6$ polynomial $h$ on $\mathbf{R}^3$ defined by
\begin{align*}
h(x_1, x_2, x_3) = (&3491640 {x_1}^4 {x_3}^2-2454945 {x_1}^6+33332535 {x_1}^4 {x_2}^2 -4145028{x_1}^4\\
&-26323635 {x_1}^2{x_2}^4+3517260 {x_1}^2{x_3}^4+30392244 {x_1}^2 {x_2}^2 \\
&-42053400 {x_1}^2 {x_2}^2 {x_3}^2-5522076 {x_1}^2 {x_3}^2 +1437395 {x_1}^2\\
&+1477725 {x_2}^6-424560 {x_3}^6-4317208 {x_2}^4+2851140 {x_2}^2 {x_3}^4\\
&+1668512 {x_3}^4+2056969 {x_2}^2+4157760 {x_2}^4 {x_3}^2-4488996 {x_2}^2 {x_3}^2\\
&-3494364 {x_3}^2) / 2701782720
\end{align*}
is harmonic on $\mathbf{R}^3$ and $\nabla h \cdot \nabla q = f$ on the ellipsoid $\{x \in \mathbf{R}^3 : q(x) = 1\}$.
\end{example}

A striking feature of the example above (and of similar examples on ellipsoids that are not spheres) is the presence of large integers in the solution even though the input data contains only small integers.

To verify that the function $h$ in Example \ref{example} really has the claimed properties, first compute the Laplacian of $h$ (use a computer unless you like arithmetic), getting $0$ (thus $h$ is harmonic, as claimed).

The next step in the verification of Example \ref{example} is to have a computer find the dot product of the gradient of $h$ and the gradient of $q$ [which is $(6x_1, 2x_2, 4x_3)$], getting a messy degree $6$ polynomial on $\mathbf{R}^3$. This messy degree $6$ polynomial is supposed to equal ${x_1}^4 {x_2}^2 - \frac{1}{315}$ on $\{x \in \mathbf{R}^3 : q(x) = 1\}$. Thus subtract ${x_1}^4 {x_2}^2 - \frac{1}{315}$ from this messy degree $6$ polynomial, getting a different messy degree $6$ polynomial that is supposed to equal $0$ on $\{x \in \mathbf{R}^3 : q(x) = 1\}$. Now ask your symbolic processing program to factor this polynomial, and then note that $q(x) - 1$ is a factor. Thus the polynomial equals $0$ on $\{x \in \mathbf{R}^3 : q(x) = 1\}$, completing the verification that $\nabla h \cdot \nabla q = f$ on $\{x \in \mathbf{R}^3 : q(x) = 1\}$. This verification provides a satisfying reassurance that the algorithm described in Section \ref{algorithm} works as expected.

The \textit{Mathematica} version of the \textit{Computing with Harmonic Functions} documentation available at \cite{A} is a live \textit{Mathematica} notebook that can be modified by the user to provide additional examples (see the \texttt{neumann} section) and to carry out the verification procedure described in the two paragraphs above.

Having verified that the function $h$ in Example \ref{example} has the claimed properties, we can note that $h$ also satisfies the equation $h(0, 0, 0) = 0$. Proposition~\ref{unique} tells us that the polynomial $h$ in Example \ref{example} is the unique harmonic function such that $h(0,0,0) = 0$ and $\nabla h \cdot \nabla q = f$ on $\{x \in \mathbf{R}^3 : q(x) = 1\}$. Thus the large integers that appear in Example \ref{example} do not arise from a nonoptimal solution of this Neumann problem---this behavior is intrinsic to the Neumann problem on ellipsoids.

The next example presents a generalized Neumann problem on the ellipsoid $\{x \in \mathbf{R}^3: 5{x_1}^2 + 3{x_2}^2 + 2{x_3}^2 = 1\}$. The input functions for this generalized Neumann problem, ${x_1}^3 {x_2}^2 x_3$ and $4{x_2}^3$, satisfy condition (a) of Theorem \ref{generalized} because both integrals in condition (a) of Theorem \ref{generalized} equal $0$ (by symmetry, because each integrand has as a factor a coordinate of $x$ raised to an odd power).
\begin{example} \label{examplegen}
Suppose $f(x_1, x_2, x_3) = {x_1}^3 {x_2}^2 x_3$, $g(x_1, x_2, x_3) = 4{x_2}^3$ and
\[
q(x) = 5{x_1}^2 + 3{x_2}^2 + 2{x_3}^2.
\]
Then the degree $6$ polynomial $h$ on $\mathbf{R}^3$ defined by
\begin{align*}
h(x_1, x_2, x_3) = &\frac{11033{x_1}^3 {x_2}^2 x_3}{806086}-\frac{4355 {x_1}^5 x_3}{3224344}-\frac{4825 {x_1}^4 x_2}{66518}-\frac{97 {x_1}^3{x_3}^3}{1612172}\\[4bp]
&-\frac{94163 {x_1}^3 x_3}{812534688}+\frac{34955 {x_1}^2{x_2}^3}{199554}-\frac{6005 {x_1}^2 x_2 {x_3}^2 }{66518}+\frac{457865 {x_1}^2 x_2}{5687289}\\[4bp]
&+\frac{716 x_1{x_3}^5}{2015215}-\frac{5437 x_1{x_2}^2 {x_3}^3}{1612172}-\frac{235273 x_1{x_3}^3}{406267344}-\frac{16629 x_1{x_2}^4 x_3}{3224344}\\[4bp]
&+\frac{564709 x_1{x_2}^2 x_3}{270844896}+\frac{1505411 x_1 x_3}{5687742816}+\frac{17278 {x_2}^5}{99777}-\frac{1049 x_2 {x_3}^4}{33259}\\[4bp]
&-\frac{1660991 {x_2}^3}{34123734}+\frac{18593 {x_2}^3 {x_3}^2}{199554}+\frac{745261 x_2 {x_3}^2}{11374578}-\frac{1621829 x_2}{34123734}
\end{align*}
satisfies the conditions $\Delta h = g$ on $\mathbf{R}^3$ and $\nabla h \cdot \nabla q = f$ on $\{x \in \mathbf{R}^3 : q(x) = 1\}$.
\end{example}

The result above is computed by using the procedure outlined by the proof of Theorem \ref{generalized} (which also requires the algorithm discussed in Section \ref{algorithm}).

A striking feature of the solution in Example \ref{examplegen} is that even though the input data contains only single-digit integers, the output includes a ten-digit integer and multiple nine-digit integers. Again, Proposition \ref{unique} implies that the polynomial $h$ given in Example \ref{examplegen} is the unique function $h$ with $h(0,0,0) = 0$ that solves this generalized Neumann problem.

For verification that the solution in Example \ref{examplegen} is correct, see the \texttt{neumann} section of the \textit{Computing with Harmonic Functions} documentation at \cite{A}.

For simplicity and clarity, we have considered in this paper only ellipsoids centered at the origin. However, the algorithm discussed in the previous section can be modified to handle ellipsoids centered at arbitrary points in $\rn$ (for such ellipsoids defined by a quadratic expression $q$, the gradient $\nabla q$ becomes slightly more complicated than considered here). The software at \cite{A} extends the algorithm discussed in this paper so that it can also handle ellipsoids not centered at the origin.

\section{Magnetic Resonance Imaging} \label{MRI}
MRI (magnetic resonance imaging) is a diagnostic tool that uses time-varying magnetic fields to produce images of anatomical structures inside the human body. The laws of physics, however, tell us that such magnetic fields will generate electric fields that can cause pain and nerve stimulation in patients undergoing imaging exams. Therefore, it is of interest to estimate the level of electric fields induced inside the imaging subject.

The Maxwell-Faraday equation in differential form is written as
\[
	\nabla \times \vect{E}\left(x,y,z,t\right) = - \frac{\partial{\vect{B}\left(x,y,z,t\right)}}{\partial{t}},
\]
which states that an applied time-varying magnetic field $\vect{B}$ will induce a spatially-varying electric field $\vect{E}$. In handling both vector fields, it is convenient to write them in terms of potentials:
\begin{equation} \label{vmp}
\vect{B} = \nabla \times \vect{A}
\end{equation}
and
\begin{equation} \label{sep}
\vect{E} = -\nabla V - \frac{\partial\vect{A}}{\partial t},
\end{equation}
where $\vect{A}$ is the vector magnetic potential and $V$ is the scalar electric potential.

The applied magnetic field $\vect{B}$ is determined by the imaging requirements such as spatial coverage and image resolution, and can be represented as a polynomial function of the spatial coordinates. From that we can specify the magnetic potential $\vect{A}$ in a polynomial form that satisfies \Ref{vmp} together with the additional constraint
\begin{equation} \label{gauge}
\nabla \cdot \vect{A}=0,
\end{equation}
which is referred to as specifying the magnetic potential using the Coulomb gauge. In order to calculate the $\vect{E}$-field from \Ref{sep}, we also need to obtain $V$, the scalar electric potential.

In MRI applications, we can assume that we are working in a quasistatic regime, which leads to the following condition:
\begin{equation} \label{continuity}
\nabla \cdot \vect{E} = 0.
\end{equation}
From this, we achieve a boundary condition that the outward normal component of the $\vect{E}$-field on the imaging subject's surface is zero:
\begin{equation} \label{enorm}
\vect{E} \cdot \n  = 0.
\end{equation}
We can now derive our working equation that we use to find $V$. Once we know what $V$ is, then it is trivial to calculate $\vect{E}$ from \Ref{sep}.

Taking the divergence of both sides of \Ref{sep} we have
\begin{align}
\nabla \cdot \vect{E} &= -\Delta V - \frac{\partial \left(\nabla \cdot \vect{A}\right)}{\partial t} \notag \\
&= -\Delta V \notag \\
& = 0, \notag
\end{align}
where the second equality comes from \Ref{gauge} and the third from \Ref{continuity}. Hence the electric potential $V$ satisfies the Laplace equation:
\begin{equation*} \label{laplaceV}
\Delta V = 0.
\end{equation*}
From \Ref{sep} and \Ref{enorm}, we also have a boundary condition for $V$ on the surface:
\begin{equation*}
\nabla V \cdot \n = -\frac{\partial\vect{A}}{\partial t} \cdot \n.
\end{equation*}
Modeling human body parts such as the torso or head as an ellipsoid (models based upon spheres have not been sufficiently accurate), the problem of calculating the induced electric field inside the body boils down to solving the Neumann problem on ellipsoids given polynomial boundary data, as discussed in this paper.

\end{document}